\documentclass[onefignum,onetabnum]{siamonline190516}

\usepackage{lipsum}
\usepackage{amsfonts}
\usepackage{graphicx}
\usepackage{epstopdf}
\usepackage{algorithmic}
\usepackage{multirow}
\ifpdf
  \DeclareGraphicsExtensions{.eps,.pdf,.png,.jpg}
\else
  \DeclareGraphicsExtensions{.eps}
\fi

\usepackage{enumitem}
\setlist[enumerate]{leftmargin=.5in}
\setlist[itemize]{leftmargin=.5in}

\newsiamremark{remark}{Remark}
\newsiamremark{assumption}{Assumption}
\newsiamremark{hypothesis}{Hypothesis}
\crefname{hypothesis}{Hypothesis}{Hypotheses}
\newsiamthm{claim}{Claim}

\newcommand{\EE}{\mathbb{E}}
\newcommand{\PP}{\mathbb{P}}
\newcommand{\VV}{\mathbb{V}}

\newcommand{\ZZ}{\mathbb{Z}}
\newcommand{\EVPI}{\mathrm{EVPI}}
\newcommand{\EVPPI}{\mathrm{EVPPI}}
\newcommand{\EVSI}{\mathrm{EVSI}}
\newcommand{\MLMC}{\mathrm{MLMC}}
\newcommand{\NMC}{\mathrm{NMC}}
\newcommand{\opt}{\mathrm{opt}}

\newcommand{\bsone}{\boldsymbol{1}}
\newcommand{\bgd}{{\overline{g_d}}}
\newcommand{\bmaxfrhod}{{\overline{\max_{d\in D}f_d(\cdot)\rho(Y\,|\,\cdot)}}}
\newcommand{\bfrhod}{{\overline{f_d(\cdot)\rho(Y\,|\,\cdot)}}}
\newcommand{\brhoY}{{\overline{\rho(Y\,|\,\cdot)}}}
\newcommand{\fracs}[2]{{\textstyle \frac{#1}{#2}}}

\allowdisplaybreaks

\headers{MLMC estimation of the expected value of sample information}{T. Hironaka, M. B. Giles, T. Goda, and H. Thom}

\title{Multilevel Monte Carlo estimation of the expected value of sample information\thanks{Submitted to the editors DATE.}}

\author{Tomohiko Hironaka\thanks{School of Engineering, University of Tokyo, Tokyo, JPN (\email{hironaka-tomohiko@g.ecc.u-tokyo.ac.jp}, \email{goda@frcer.t.u-tokyo.ac.jp}).}
\and Michael B. Giles\thanks{Mathematical Institute, University of Oxford, Oxford, UK (\email{mike.giles@maths.ox.ac.uk}).}
\and Takashi Goda\footnotemark[2]
\and Howard Thom\thanks{Bristol Medical School, University of Bristol, Bristol, UK (\email{howard.thom@bristol.ac.uk}).}
}

\usepackage{amsopn}

\begin{document}

\maketitle

\begin{abstract}
We study Monte Carlo estimation of the expected value of sample information (EVSI) which measures the expected benefit of gaining additional information for decision making under uncertainty. EVSI is defined as a nested expectation in which an outer expectation is taken with respect to one random variable $Y$ and an inner conditional expectation with respect to the other random variable $\theta$. Although the nested (Markov chain) Monte Carlo estimator has been often used in this context, a root-mean-square accuracy of $\varepsilon$ is achieved notoriously at a cost of $O(\varepsilon^{-2-1/\alpha})$, where $\alpha$ denotes the order of convergence of the bias and is typically between $1/2$ and $1$. In this article we propose a novel efficient Monte Carlo estimator of EVSI by applying a multilevel Monte Carlo (MLMC) method. Instead of fixing the number of inner samples for $\theta$ as done in the nested Monte Carlo estimator, we consider a geometric progression on the number of inner samples, which yields a hierarchy of estimators on the inner conditional expectation with increasing approximation levels. Based on an elementary telescoping sum, our MLMC estimator is given by a sum of the Monte Carlo estimates of the differences between successive approximation levels on the inner conditional expectation. We show, under a set of assumptions on decision and information models, that successive approximation levels are tightly coupled, which directly proves that our MLMC estimator improves the necessary computational cost to optimal $O(\varepsilon^{-2})$. Numerical experiments confirm the considerable computational savings as compared to the nested Monte Carlo estimator.
\end{abstract}



\section{Introduction}
Motivated by applications to medical decision making under uncertainty \cite{WSCAAbook}, we study Monte Carlo estimation of the expected value of sample information (EVSI). Let $\theta$ be a vector of random variables representing the uncertainty in the effectiveness of different medical treatments. Let $D$ be a finite set of possible medical treatments, and for each treatment $d\in D$, $f_d$ denotes a function of $\theta$ representing some measure of the patient outcome with ``larger the better'', where quality-adjusted life years (QALY) is typically employed in the context of medical decision making \cite{ALC04,BKOC07,SOBB15,HMB16}. Without any knowledge about $\theta$, the best treatment is the one which maximizes the expectation of $f_d$, giving the average outcome:
\begin{align}\label{eq:prior_ev}
 \max_{d\in D} \EE_\theta\left[ f_d(\theta)\right],
\end{align}
where $\EE_{\theta}[\cdot]$ denotes the expectation taken with respect to the prior probability density function of $\theta$. On the other hand, if perfect information on $\theta$ is available, the best treatment after knowing the value of $\theta$ is simply the one which maximizes $f_d(\theta)$, so that, on average, the outcome will be
\[ \EE_\theta\left[ \max_{d\in D} f_d(\theta)\right]. \]
The difference between these two values is called the expected value of perfect information (EVPI):
\[ \EVPI := \EE_\theta\left[ \max_{d\in D} f_d(\theta)\right]- \max_{d\in D} \EE_\theta\left[ f_d(\theta)\right]. \]
However, it will be rare that we have access to perfect information on $\theta$. In practice, what we obtain, for instance, through carrying out some new medical research is either \emph{partial perfect information} or \emph{sample information} on $\theta$.

Partial perfect information is nothing but perfect information only on a subset of random variables $\theta_1$ for a partition $\theta=(\theta_1,\theta_2)$, where $\theta_1$ and $\theta_2$ are assumed independent. After knowing the value of $\theta_1$, the best treatment is the one which maximizes the partial expectation $\EE_{\theta_2}\left[ f_d(\theta_1,\theta_2)\right]$. Therefore, the average outcome with partial perfect information on $\theta_1$ will be
\[ \EE_{\theta_1}\left[ \max_{d\in D}\EE_{\theta_2}\left[ f_d(\theta_1,\theta_2)\right] \right], \]
and the increment from \eqref{eq:prior_ev} is called the expected value of partial perfect information (EVPPI):
\[ \EVPPI := \EE_{\theta_1}\left[ \max_{d\in D}\EE_{\theta_2}\left[ f_d(\theta_1,\theta_2)\right] \right]- \max_{d\in D} \EE_\theta\left[ f_d(\theta)\right]. \]

Sample information on $\theta$, which is of our interest in this article, is a single realization drawn from some probability distribution. To be more precise, we consider that information $Y$ is stochastically generated according to the forward information model:
\begin{align}\label{eq:information_model}
 Y = h(\theta)+\epsilon,
\end{align}
where $h$ is a known deterministic function of $\theta$ possibly with multiple outputs and $\epsilon$ is a zero-mean random variable with density $\rho$. Note that $h$ and $\epsilon$ are called the observation operator and the observation noise, respectively \cite[Section~2]{S10}. It is widely known that Bayes' theorem provides an update of the probability density of $\theta$ after observing $Y$:
\begin{align}\label{eq:Bayes}
 \pi^Y(\theta) = \frac{\rho(Y\,|\, \theta)\pi_0(\theta)}{\EE_\theta[\rho(Y\,|\, \theta)]},
\end{align}
where $\pi_0(\theta)$ denotes the prior probability density of $\theta$, $\rho(Y\,|\, \theta)$ the conditional probability density of $Y$ given $\theta$. Here $\rho(Y\,|\, \theta)$ is also called the likelihood of the information $Y$, and it follows from the model \eqref{eq:information_model} that $\rho(Y\,|\, \theta):=\rho(Y-h(\theta))$.

Now, if such sample information $Y$ is available, by choosing the best treatment which maximizes the conditional expectation $\EE_{\theta\,|\, Y}\left[ f_d(\theta)\right]$ depending on $Y$, where $\EE_{\theta\,|\, Y}[\cdot]$ denotes the expectation taken with respect to the conditional probability density $\pi^Y(\theta)$, the overall average outcome becomes
\[ \EE_{Y}\left[ \max_{d\in D}\EE_{\theta\,|\, Y}\left[ f_d(\theta)\right] \right] . \]
Then EVSI represents the expected benefit of gaining the information $Y$ and is defined by the difference
\[ \EVSI := \EE_{Y}\left[ \max_{d\in D}\EE_{\theta\,|\, Y}\left[ f_d(\theta)\right] \right] - \max_{d\in D} \EE_\theta\left[ f_d(\theta)\right] .\]

In this article we are concerned with Monte Carlo estimation of EVSI. Given that EVPI can be estimated with root-mean-square accuracy $\varepsilon$ by using $N=O(\varepsilon^{-2})$ i.i.d.\ samples of $\theta$, denoted by $\theta^{(1)},\ldots,\theta^{(N)}$, as
\[ \frac{1}{N}\sum_{n=1}^{N}\max_{d\in D} f_d(\theta^{(n)}) - \max_{d\in D} \frac{1}{N}\sum_{n=1}^{N}f_d(\theta^{(n)}), \]
it suffices to efficiently estimate the difference between EVPI and EVSI:
\begin{align}\label{eq:difference}
 \EVPI-\EVSI = \EE_\theta\left[ \max_{d\in D} f_d(\theta)\right]-\EE_{Y}\left[ \max_{d\in D}\EE_{\theta\,|\, Y}\left[ f_d(\theta)\right] \right]  .
\end{align}
Because of the non-commutativity between the operators $\EE$ and $\max_{d\in D}$, this estimation is inherently a nested expectation problem and it is far from trivial whether we can construct a good Monte Carlo estimator which achieves a root-mean-square accuracy $\varepsilon$ at a cost of $O(\varepsilon^{-2})$.

Classically the most standard approach is to apply nested (Markov chain) Monte Carlo methods. For $M, N\in \ZZ_{>0}$, let $Y^{(1)},\ldots,Y^{(N)}$ be $N$ outer i.i.d.\ samples of $Y$, and for each $1\leq n\leq N$, let $\theta^{(n,1)},\ldots,\theta^{(n,M)}$ be $M$ inner i.i.d.\ samples of $\theta$ conditional on $Y^{(n)}$. Then the nested Monte Carlo estimator of $\EVPI-\EVSI$ is given by
\begin{align}\label{eq:nmc}
 \frac{1}{N}\sum_{n=1}^{N}\left[ \frac{1}{M}\sum_{m=1}^{M}\max_{d\in D}f_d(\theta^{(n,m)})-\max_{d\in D}\frac{1}{M}\sum_{m=1}^{M}f_d(\theta^{(n,m)})\right].
\end{align}
Here it is often hard to generate inner i.i.d.\ samples of $\theta$ conditional on some value of $Y$ directly (although, conversely, it is quite easy to generate i.i.d.\ samples of $Y$ conditional on some value of $\theta$ according to \eqref{eq:information_model}). This is a major difference from estimating EVPPI.

To work around this difficulty, although the resulting samples are no longer i.i.d., one relies on Markov chain Monte Carlo (MCMC) sampling techniques such as Metropolis-Hasting sampling and Gibbs sampling, see \cite{Lbook,RCbook}. Under certain conditions, it follows from \cite{JO10,LMN13} that one can establish a non-asymptotic error bound of $O(M^{-1/2})$ for MCMC estimation of the inner conditional expectation. Still, as inferred from a recent work of Giles \& Goda \cite{GG19} on EVPPI estimation, we need $N=O(\varepsilon^{-2})$ and $M=O(\varepsilon^{-1/\alpha})$ samples for outer and inner expectations, respectively, to estimate $\EVPI-\EVSI$ with root-mean-square accuracy $\varepsilon$. Here $\alpha$ denotes the order of convergence of the bias and is typically between $1/2$ and $1$. This way the necessary total computational cost is of $O(\varepsilon^{-2-1/\alpha})$.

In this article, building upon the earlier work by Giles \& Goda \cite{GG19}, we develop a novel efficient Monte Carlo estimator of $\EVPI-\EVSI$ by using a multilevel Monte Carlo (MLMC) method \cite{G08,G15}. Although there has been extensive recent research on efficient approximations of EVSI in the medical decision making context \cite{SOBB15,HMB16,M16,JA18}, our proposal avoids function approximations on the inner conditional expectation and any reliance on assumptions of multilinearity of $f_d$ or weak correlation between random variables in $\theta$. Recently MLMC estimators have been studied intensively for nested expectations of different forms, for instance, by \cite{BHR15,GH19,GHIxx}. We also refer to \cite{G18} for a review of recent developments of MLMC applied to nested expectation problems. Importantly, our approach developed in this article does not require MCMC sampling for generating inner conditional samples of $\theta$ and can achieve a root-mean-square accuracy $\varepsilon$ at a cost of optimal $O(\varepsilon^{-2})$. Moreover, it is straightforward to incorporate importance sampling techniques within our estimator, which may sometimes reduce the variance of the estimator significantly.

\section{Multilevel Monte Carlo}\label{sec:mlmc}
\subsection{Basic theory}\label{subsec:basic}
Before introducing our estimator of $\EVPI-\EVSI$, we give an overview of the MLMC method. Let $P$ be a real-valued random variable which cannot be sampled exactly, and let $P_0,P_1,\ldots$ be a sequence of real-valued random variables which approximate $P$ with increasing accuracy but also with increasing cost. In order to estimate $\EE[P]$, we first approximate $\EE[P]$ by $\EE[P_L]$ for some $L\in \ZZ_{\geq 0}$ and then the standard Monte Carlo method estimates $\EE[P_L]$ by using i.i.d.\ samples $P_L^{(1)}, P_L^{(2)},\ldots$ of $P_L$ as
\[ \EE[P]\approx \EE[P_L]\approx \overline{P_L}^{N} := \frac{1}{N}\sum_{n=1}^{N}P_L^{(n)}. \]

On the other hand, the MLMC method exploits the following telescoping sum representation:
\[ \EE[P_L] = \EE[P_0] + \sum_{\ell=1}^{L}\EE[P_\ell-P_{\ell-1}]. \]
More generally, given a sequence of random variables $\Delta P_0,\Delta P_1,\ldots$ which satisfy
\[ \EE[\Delta P_0] = \EE[P_0]\qquad \text{and} \qquad \EE[\Delta P_\ell] = \EE[P_\ell-P_{\ell-1}] \quad \text{for $\ell\geq 1$},\]
we have
\[ \EE[P_L] = \sum_{\ell=0}^{L}\EE[\Delta P_\ell]. \]
Then the MLMC estimator is given by a sum of independent Monte Carlo estimates of $\EE[\Delta P_0], \EE[\Delta P_1], \ldots$, i.e.,
\begin{align}\label{eq:mlmc}
 Z_{\MLMC} = \sum_{\ell=0}^{L}\overline{\Delta P}^{N_\ell}_\ell = \sum_{\ell=0}^{L}\frac{1}{N_\ell}\sum_{n=1}^{N_\ell}\Delta P_\ell^{(n)}.
\end{align}
Since $P_0,P_1,\ldots$ approximate $P$ with increasing accuracy, through a tight coupling of $P_{\ell-1}$ and $P_\ell$, the variance of the correction variable $\Delta P_\ell$ is expected to get smaller as the level $\ell$ increases. This implies that the numbers of samples $N_0,N_1,\ldots$ can also get smaller as the level $\ell$ increases so as to estimate each quantity $\EE[\Delta P_\ell]$ accurately. If this is the case, the total computational cost can be reduced significantly as compared to the standard Monte Carlo method.

The following basic theorem from \cite{G08,CGST11,G15} makes the above observation explicit.
\begin{theorem}\label{thm:mlmc}
Let $P$ be a random variable, and for $\ell\in \ZZ_{\geq 0}$, let $P_\ell$ be the level $\ell$ approximation of $P$. Assume that there exist independent correction random variables $\Delta P_{\ell}$ with expected cost $C_\ell$ and variance $V_\ell$, and positive constants $\alpha,\beta,\gamma,c_1,c_2,c_3$ such that $\alpha \geq \min(\beta,\gamma)/2$ and
\begin{enumerate}
\item $\EE[\Delta P_{\ell}] = \begin{cases} \EE[P_0], & \ell=0 \\ \EE[P_\ell-P_{\ell-1}] , & \ell\geq 1\end{cases}$
\item $|\EE[P_\ell-P]|\leq c_12^{-\alpha \ell}$,
\item $V_\ell \leq c_2 2^{-\beta \ell}$,
\item $C_\ell \leq c_3 2^{\gamma \ell}$.
\end{enumerate}
Then there exists a positive constant $c_4$ such that, for any root-mean-square accuracy $\varepsilon<e^{-1}$, there are $L$ and $N_0,\ldots,N_L$ for which the MLMC estimator \eqref{eq:mlmc} achieves a mean-square error less than $\varepsilon^2$, i.e.,
\[ \EE[(Z_{\MLMC}-\EE[P])^2] \leq \varepsilon^2 \]
with a computational cost $C$ bounded above by
\[ \EE[C] \leq \begin{cases} c_4\varepsilon^{-2}, & \beta>\gamma, \\ c_4\varepsilon^{-2}(\log \varepsilon^{-1})^2, & \beta=\gamma, \\ c_4\varepsilon^{-2-(\gamma-\beta)/\alpha}, & \beta<\gamma. \end{cases} \]
\end{theorem}

\begin{remark}
As discussed in \cite[Section~2.1]{CGST11} and \cite[Section~2.1]{GG19}, under similar assumptions to those in Theorem~\ref{thm:mlmc}, the standard Monte Carlo estimator achieves a root-mean-square accuracy at a cost of $O(\varepsilon^{-2-\gamma/\alpha})$. Therefore, regardless of the values of $\beta>0$ and $\gamma$, the MLMC estimator always has an asymptotically lower complexity bound than the standard Monte Carlo estimator.
\end{remark}

\subsection{MLMC estimator}\label{subsec:estimator}
Here we construct an MLMC estimator of the difference $\EVPI-\EVSI$. Our starting point is to insert \eqref{eq:Bayes} into \eqref{eq:difference}, which results in
\begin{align*}
 \EVPI-\EVSI & = \EE_Y \EE_{\theta\,|\, Y}\left[ \max_{d\in D} f_d(\theta)\right]-\EE_{Y}\left[ \max_{d\in D}\EE_{\theta\,|\, Y}\left[ f_d(\theta)\right] \right] \\
 & = \EE_Y\left[ \frac{\EE_\theta\left[ \max_{d\in D} f_d(\theta)\rho(Y\,|\, \theta)\right]}{\EE_\theta[\rho(Y\,|\, \theta)]} - \max_{d\in D}\frac{\EE_{\theta}\left[ f_d(\theta)\rho(Y\,|\, \theta)\right]}{\EE_\theta[\rho(Y\,|\, \theta)]} \right] .
\end{align*}

This has converted the posterior expectation with respect to $\theta$ given $Y$ into the ratio of two prior expectations with respect to $\theta$, which avoids the need for MCMC sampling. This idea has been used not only in the current context \cite{M16} but also in other areas related to Bayesian computations, see \cite{SS13,SST17,DGLS18,GP18} among many others. On a technical level, this gives a decisive difference from estimating EVPPI as considered in \cite{GG19} for which we do not need such treatment. As discussed later in Subsection~\ref{subsec:aux}, efficient estimation of the nested expectation involving the ratio of two expectations is far from trivial. Although the algorithmic idea may perhaps seem a natural extension from \cite{GG19} the numerical analysis of it is certainly not.

Based on the above expression, within the framework of the MLMC method, let us consider a real-valued random variable
\[ P = \frac{\EE_\theta\left[ \max_{d\in D} f_d(\theta)\rho(Y\,|\, \theta)\right]}{\EE_\theta[\rho(Y\,|\, \theta)]} - \max_{d\in D}\frac{\EE_{\theta}\left[ f_d(\theta)\rho(Y\,|\, \theta)\right]}{\EE_\theta[\rho(Y\,|\, \theta)]} \]
with $Y$ being the underlying random variable. We see that $\EE_Y[P] = \EVPI-\EVSI$ but $P$ cannot be sampled exactly because of the inner expectations.

It is important, however, that all these inner expectations are taken with respect to the prior probability density of $\theta$, so that they can be approximated by the standard Monte Carlo method without requiring any MCMC sampling. This way, as a sequence of random variables $P_0,P_1,\ldots$ which approximate $P$ with increasing accuracy, we consider the standard Monte Carlo estimation of $P$:
\[ P_\ell := \frac{\bmaxfrhod^{M_\ell}}{\brhoY^{M_\ell}}- \max_{d\in D}\frac{\bfrhod^{M_\ell}}{\brhoY^{M_\ell}}\]
for an increasing sequence $M_0<M_1<\cdots$ with $M_\ell \to \infty$ as $\ell\to \infty$. Here, in the definition of $P_\ell$, we can use either the same $M_\ell$ samples of $\theta$ for all of the averages, or independent $M_\ell$ samples of $\theta$ for each average. In this article we focus on the former approach and consider a geometric progression for $M_0,M_1,\ldots$, i.e., let $M_\ell=M_02^{\ell}$ for some $M_0\in \ZZ_{>0}$. In fact, the subsequent argument holds for a general integer base $b\geq 2$, but we restrict ourselves to the case $b=2$ for simplicity of presentation.

Regarding a sequence of the correction variables $\Delta P_0,\Delta P_1,\ldots$, following the ideas of \cite{BHR15,GG19,GH19}, we consider an antithetic coupling of $P_{\ell-1}$ and $P_\ell$. That is, the set of $M_02^{\ell}$ samples of $\theta$ used to compute $P_\ell$ is split into two disjoint sets of $M_02^{\ell-1}$ samples to compute two independent realizations of $P_{\ell-1}$, denoted by $P_{\ell-1}^{(a)}$ and $P_{\ell-1}^{(b)}$. Then $\Delta P_\ell$ is defined by $\Delta P_0 = P_0$ and
\begin{align*}
\Delta P_\ell & := P_\ell-\frac{P_{\ell-1}^{(a)}+P_{\ell-1}^{(b)}}{2} \\
& = \frac{\bmaxfrhod}{\brhoY} - \max_{d\in D}\frac{\bfrhod}{\brhoY} \\
&\quad - \frac{1}{2}\left[ \frac{\bmaxfrhod^{(a)}}{\brhoY^{(a)}}+\frac{\bmaxfrhod^{(b)}}{\brhoY^{(b)}}\right] \\
& \quad+ \frac{1}{2}\left[ \max_{d\in D}\frac{\bfrhod^{(a)}}{\brhoY^{(a)}}+\max_{d\in D}\frac{\bfrhod^{(b)}}{\brhoY^{(b)}}\right] ,
\end{align*}
for $\ell\geq 1$, where we have omitted the superscripts $M_\ell$ from the first and second terms, and the averages with the superscripts $(a)$ and $(b)$ are taken by using the first and second $M_02^{\ell-1}$ samples of $\theta$ used to compute the first two terms, respectively. Assuming that each computation of $f_d(\theta)$, $\rho(Y\,|\, \theta)$ and $h(\theta)$ can be performed with unit cost, it is clear that $\gamma=1$ in Theorem~\ref{thm:mlmc} and $\EE[\Delta P_{\ell}] = \EE[P_\ell-P_{\ell-1}]$ for $\ell\geq 1$ because of the independence of the samples.

We mean by the word ``antithetic'' that the following properties hold:
\begin{align}
\brhoY & = \frac{\brhoY^{(a)}+\brhoY^{(b)}}{2}, \label{eq:anti1}\\
\bfrhod & =  \frac{\bfrhod^{(a)}+\bfrhod^{(b)}}{2}\quad \text{for all $d\in D$}, \label{eq:anti2} \\
\bmaxfrhod & = \frac{\bmaxfrhod^{(a)}+\bmaxfrhod^{(b)}}{2}. \label{eq:anti3}
\end{align}
This is the key advantage of the antithetic correction as compared to the standard correction $P_\ell-P_{\ell-1}^{(a)}$.

\subsection{Combination with importance sampling}\label{subsec:is}
In practical applications, it is often the case where the likelihood $\rho(Y\,|\, \theta)$ (as a function of $\theta$ for a fixed $Y$) is highly concentrated around some values of $\theta$. If one uses the i.i.d.\ samples of $\theta$ following from the prior density $\pi_0(\theta)$ for estimating $\EE_\theta[\rho(Y\,|\,\theta)]$ and $\EE_\theta[f_d(\theta)\rho(Y\,|\,\theta)]$, most of the samples can be distributed outside such concentrated regions. As a result, these quantities are estimated as almost zero, yielding a numerical instability of the estimates. Regarding theoretical analyses on such concentrated posterior measures on $\theta$, we refer to some recent papers \cite{SS16,SSWxx}.

As discussed in \cite[Section~3.4]{GHIxx} for MLMC applied to another nested expectation problem, one can combine importance sampling techniques with our MLMC estimator to address this issue. Let $q^Y(\theta)$ be an importance distribution of $\theta$ conditional on $Y$, which needs to satisfy $q^Y(\theta)>0$ for any $\theta$ with $\pi_0(\theta)>0$. Since we have
\[ \EE_\theta[f_d(\theta)\rho(Y\,|\,\theta)] = \EE_{\theta\sim q^Y}\left[f_d(\theta)\frac{\rho(Y\,|\,\theta)\pi_0(\theta)}{q^Y(\theta)}\right] \]
and
\[ \EE_\theta[\rho(Y\,|\,\theta)] = \EE_{\theta\sim q^Y}\left[\frac{\rho(Y\,|\,\theta)\pi_0(\theta)}{q^Y(\theta)}\right], \]
the random variables $P_\ell$ and $\Delta P_\ell$ can be replaced, respectively, by
\begin{align*}
P_\ell & = \frac{\overline{\max_{d\in D}f_d(\cdot)\rho(Y\,|\, \cdot)\pi_0(\cdot)/q^Y(\cdot)}^{M_\ell}}{\overline{\rho(Y\,|\, \cdot)\pi_0(\cdot)/q^Y(\cdot)}^{M_\ell}} - \max_{d\in D}\frac{\overline{f_d(\cdot)\rho(Y\,|\, \cdot)\pi_0(\cdot)/q^Y(\cdot)}^{M_\ell}}{\overline{\rho(Y\,|\, \cdot)\pi_0(\cdot)/q^Y(\cdot)}^{M_\ell}}, \\
\Delta P_\ell & = \frac{\overline{\max_{d\in D}f_d(\cdot)\rho(Y\,|\, \cdot)\pi_0(\cdot)/q^Y(\cdot)}}{\overline{\rho(Y\,|\, \cdot)\pi_0(\cdot)/q^Y(\cdot)}} - \max_{d\in D}\frac{\overline{f_d(\cdot)\rho(Y\,|\, \cdot)\pi_0(\cdot)/q^Y(\cdot)}}{\overline{\rho(Y\,|\, \cdot)\pi_0(\cdot)/q^Y(\cdot)}} \\
& \quad - \frac{1}{2}\left[ \frac{\overline{\max_{d\in D}f_d(\cdot)\rho(Y\,|\, \cdot)\pi_0(\cdot)/q^Y(\cdot)}^{(a)}}{\overline{\rho(Y\,|\, \cdot)\pi_0(\cdot)/q^Y(\cdot)}^{(a)}}+\frac{\overline{\max_{d\in D}f_d(\cdot)\rho(Y\,|\, \cdot)\pi_0(\cdot)/q^Y(\cdot)}^{(b)}}{\overline{\rho(Y\,|\, \cdot)\pi_0(\cdot)/q^Y(\cdot)}^{(b)}}\right] \\
& \quad + \frac{1}{2}\left[ \max_{d\in D}\frac{\overline{f_d(\cdot)\rho(Y\,|\, \cdot)\pi_0(\cdot)/q^Y(\cdot)}^{(a)}}{\overline{\rho(Y\,|\, \cdot)\pi_0(\cdot)/q^Y(\cdot)}^{(a)}}+\max_{d\in D}\frac{\overline{f_d(\cdot)\rho(Y\,|\, \cdot)\pi_0(\cdot)/q^Y(\cdot)}^{(b)}}{\overline{\rho(Y\,|\, \cdot)\pi_0(\cdot)/q^Y(\cdot)}^{(b)}}\right] ,
\end{align*}
where all of the averages are taken with respect to i.i.d.\ samples of $\theta\sim q^Y$ for a randomly chosen $Y$.

We suggest to find a good approximation of the posterior distribution $\pi^Y(\theta)$ for $q^Y(\theta)$. If one can do so, it follows from Bayes' theorem \eqref{eq:Bayes} that
\[ \frac{\rho(Y\,|\,\theta)\pi_0(\theta)}{q^Y(\theta)} \approx \frac{\rho(Y\,|\,\theta)\pi_0(\theta)}{\pi^Y(\theta)}= \EE_\theta[\rho(Y\,|\, \theta)]. \]
Since the right-most side does not depend on $\theta$, the integrand $\rho(Y\,|\,\theta)\pi_0(\theta)/q^Y(\theta)$ appearing in the denominator of each term for $P_\ell$ and $\Delta P_\ell$ becomes close to a constant function, so that its variance is extremely small. This way we can avoid the numerical instability of our original MLMC estimator.

\section{Theoretical results}\label{sec:theory}
In this section, we prove $\beta>\gamma=1$ and $\alpha\geq \beta/2$ under a set of assumptions on the decision and information models. This directly implies from Theorem~\ref{thm:mlmc} that our MLMC estimator of $\EVPI-\EVSI$ achieves a root-mean-square accuracy $\varepsilon$ at a cost of optimal $O(\varepsilon^{-2})$.

In what follows, for simplicity of notation, we write $\rho(Y)=\EE_\theta[\rho(Y\,|\, \theta)]$ and $F_d(Y)=\EE_{\theta}\left[ f_d(\theta)\rho(Y\,|\, \theta)\right]$. Note that we have
\[ \EE\left[\frac{\overline{\rho(Y\,|\, \cdot)}}{\rho(Y)}\,|\, Y\right] = \EE_\theta\left[\frac{\rho(Y\,|\, \theta)}{\rho(Y)}\,|\, Y\right] = 1. \]
Moreover we write
\[ \bgd^{(a)} = \frac{\bfrhod^{(a)}}{\brhoY^{(a)}} ,\ \bgd^{(b)} = \frac{\bfrhod^{(b)}}{\brhoY^{(b)}},\ \bgd = \frac{\bfrhod}{\brhoY}, \]
for $d\in D$, and
\[ \overline{g}_{\max}^{(a)} = \frac{\bmaxfrhod^{(a)}}{\brhoY^{(a)}},\ \overline{g}_{\max}^{(b)} = \frac{\bmaxfrhod^{(b)}}{\brhoY^{(b)}},\ \overline{g}_{\max} = \frac{\bmaxfrhod}{\brhoY}. \]
Then $\Delta P_\ell$ is given by $\Delta P_\ell=\Delta P_{\ell,1}+\Delta P_{\ell,2}$ with
\[ \Delta P_{\ell,1}=\overline{g}_{\max}-\frac{1}{2}\left(\overline{g}_{\max}^{(a)}+\overline{g}_{\max}^{(b)}\right) \quad \text{and}\quad \Delta P_{\ell,2}=\frac{1}{2}\left( \max_{d\in D}\bgd^{(a)}+ \max_{d\in D}\bgd^{(b)} \right)-\max_{d\in D}\bgd.\]

For a given $Y$, we define
\[ G_d(Y) := \EE_{\theta\,|\, Y}\left[ f_d(\theta)\right] = \frac{\EE_{\theta}\left[ f_d(\theta)\rho(Y\,|\, \theta)\right]}{\EE_\theta[\rho(Y\,|\, \theta)]} =  \frac{F_d(Y)}{\rho(Y)}  \]
for each $d\in D$, and
\[ d_{\opt}(Y) := \arg\max_{d\in D}F_d(Y)=\arg\max_{d\in D}G_d(Y). \]
The second equality is trivial since the denominator of $G_d$ does not affect the choice $\max_{d\in D}$. The domain for $Y$ is divided into a number of sub-domains in which the optimal decision $d_{\opt}$ is unique. We denote by $K$ the dividing decision manifold on which $d_{\opt}$ is not uniquely defined and we assume that $K$ is a lower-dimensional subspace of the domain for $Y$.

Let us give some assumptions on the decision and information models:
\begin{assumption}\label{assm:1}
There exists a constant $0<F_{max}<\infty$ such that $|f_d(\theta)| \leq F_{max}$ for any $d$ and $\theta$.
\end{assumption}

\begin{assumption}\label{assm:2}
There exist a constant $c_0>0$ such that for all $0<\epsilon < 1$
\[\PP_Y\left[ \min_{y\in K}\| Y-y\| \leq \epsilon \right]\leq c_0\epsilon. \]
\end{assumption}

\begin{assumption}\label{assm:3}
There exist constants $c_1,c_2>0$ such that if $Y\notin K$, the following holds:
\[ \max_{d\in D}G_d(Y) - \max_{\substack{d\in D\\ d\neq d_\opt(Y)}}G_d(Y) > \min\left( c_1,c_2\min_{y\in K}\| Y-y\| \right) .\]
\end{assumption}

\begin{assumption}\label{assm:4}
There exists a constant $p\geq 2$ such that
\[ \EE_Y \left[ \EE_\theta\left[ \left(\frac{\rho(Y\,|\, \theta)}{\rho(Y)}\right)^p \right] \right] < \infty .\]
\end{assumption}

\begin{remark}
Assumptions~\ref{assm:1}--\ref{assm:3} are similar to those considered in \cite{GG19}. In particular, Assumption~\ref{assm:2} is introduced to ensure a bound on the probability that $Y$ is close to the decision manifold $K$, while Assumption~\ref{assm:3} is to ensure a linear separation of different decisions as $Y$ moves away from $K$. Assumption~\ref{assm:1}, which is stronger than \cite[Assumption~1]{GG19}, together with Assumption~\ref{assm:4} enables us to bound the difference between $\bgd^{(a)}, \bgd^{(b)}, \bgd$ and $G_d$.
\end{remark}

Now we are ready to state our main result of this article.
\begin{theorem}\label{thm:main}
If Assumptions~\ref{assm:1}--\ref{assm:4} are satisfied, we have
\[ \VV[\Delta P_\ell] = O(2^{-(3p/(2p+2))\ell})\quad \text{and} \quad \EE[|\Delta P_\ell|] = O(2^{-(p/(p+1))\ell}), \]
where $p\geq 2$ is as given in Assumption~\ref{assm:4}.
\end{theorem}

\begin{remark}
When an importance sampling is used within the MLMC estimator, the same orders of the variance and the mean of $\Delta P_\ell$ can be shown by replacing Assumption~\ref{assm:4} with the existence of a constant $p\geq 2$ such that
\[ \EE_Y \left[ \EE_\theta\left[ \left(\frac{\rho(Y\,|\, \theta)\pi_0(\theta)}{\rho(Y)q^Y(\theta)}\right)^p \right] \right] < \infty. \]
Since the result can be proven in the same manner with the original MLMC estimator, we shall give a proof of Theorem~\ref{thm:main} only for the original estimator. 
\end{remark}

This theorem implies that the parameters $\alpha$ and $\beta$ in Theorem~\ref{thm:mlmc} are equal to $p/(p+1)$ and $3p/(2p+2)$, respectively. Since $\gamma=1<\beta$ if $p>2$, our MLMC estimator of $\EVPI-\EVSI$ is in the first regime. Therefore, the total computational complexity to achieve a root-mean-square accuracy $\varepsilon$ is of order $\varepsilon^{-2}$. If $p=2$, on the other hand, the equality $\beta=\gamma$ holds, which means that our MLMC estimator is in the second regime. In the next subsection, we give a proof of this theorem by using several lemmas which are shown later in Subsection~\ref{subsec:aux}.

\subsection{Proof of the main result}
We follow a similar argument to that used in \cite[Theorem~3]{GG19} in conjunction with novel results shown later. Recalling that $\Delta P_\ell=\Delta P_{\ell,1}+\Delta P_{\ell,2}$, we have
\[ \VV[\Delta P_\ell] \leq \EE[|\Delta P_\ell|^2] \leq 2\EE[|\Delta P_{\ell,1}|^2] +2\EE[|\Delta P_{\ell,2}|^2]. \]
We shall see later in Remark~\ref{rem:gmax} that the first term on the right-hand side is of $O(2^{-\min(p,4)\ell/2})$ which decays no slower than the desired order $2^{-(3p/(2p+2))\ell}$ for any $p\geq 2$. Thus it suffices to prove that the second term is of $O(2^{-(3p/(2p+2))\ell})$.

For $p$ as given in Assumption~\ref{assm:4}, let us define $\epsilon = 2^{-(p/(2p+2))\ell}$ and consider the events
\begin{align*}
A & \equiv \left\{ \min_{y\in K}\| Y-y\| \leq \epsilon \right\}, \\
B & \equiv \bigcup_{d\in D}\left\{ \max\left(  |\overline{g_d}^{(a)}-G_d|,  |\overline{g_d}^{(b)}-G_d|,  |\overline{g_d}-G_d|\right) \geq \frac{1}{2}c_2\epsilon \right\},
\end{align*}
where $c_2$ is as defined in Assumption~\ref{assm:3}.

For an event $E$, let $\bsone_E$ denote the indicator function which is 1 if $\omega\in E$, and zero otherwise, and let $E^c$ denote the complement of $E$. By using H\"{o}lder's inequality, we have
\begin{align*}
\EE[|\Delta P_{\ell,2}|^2] & = \EE[|\Delta P_{\ell,2}|^2\bsone_{A\cup B}] +\EE[|\Delta P_{\ell,2}|^2\bsone_{A^c\cap B^c}] \\
& \leq (\EE[|\Delta P_{\ell,2}|^p])^{2/p}(\EE[\bsone_{A\cup B}^{p/(p-2)}])^{(p-2)/p} +\EE[|\Delta P_{\ell,2}|^2\bsone_{A^c\cap B^c}] \\
& \leq (\EE[|\Delta P_{\ell,2}|^p)^{2/p}(\PP[A]+\PP[B])^{(p-2)/p} +\EE[|\Delta P_{\ell,2}|^2\bsone_{A^c\cap B^c}].
\end{align*}
In the following we show bounds on $\EE[|\Delta P_{\ell,2}|^p]$, $\PP[A]+\PP[B]$ and $\EE[|\Delta P_{\ell,2}|^2\bsone_{A^c\cap B^c}]$, respectively.\vspace{5pt}

\textbf{Bound on $\EE[|\Delta P_{\ell,2}|^p]$.} Noting that the inequality
\[ |\max_{d\in D}a_d-\max_{d\in D}b_d| \leq \max_{d\in D}|a_d-b_d|\leq \sum_{d\in D}|a_d-b_d| \]
holds for any two $|D|$-dimensional vectors with component $a_d,b_d$ and applying Jensen's inequality twice, we get
\begin{align*}
|\Delta P_{\ell,2}|^p & = \left| \frac{1}{2}\left( \max_{d\in D}\bgd^{(a)}+ \max_{d\in D}\bgd^{(b)} \right)-\max_{d\in D}\bgd \right|^p \\
& = \left| \frac{1}{2}\left( \max_{d\in D}\bgd^{(a)}-\max_{d\in D}G_d\right) +\frac{1}{2}\left( \max_{d\in D}\bgd^{(b)}-\max_{d\in D}G_d\right) - \left( \max_{d\in D}\bgd-\max_{d\in D}G_d\right) \right|^p \\
& \leq \left| \sum_{d\in D}\left(\frac{1}{2}\left| \bgd^{(a)}-G_d\right| +\frac{1}{2}\left| \bgd^{(b)}-G_d\right| + \left| \bgd-G_d\right|\right) \right|^p \\
& \leq |D|^{p-1}\sum_{d\in D}\left( \frac{1}{2}\left| \bgd^{(a)}-G_d\right| +\frac{1}{2}\left| \bgd^{(b)}-G_d\right| + \left| \bgd-G_d\right|\right)^p \\
& \leq (2|D|)^{p-1}\sum_{d\in D}\left( \frac{1}{2}\left| \bgd^{(a)}-G_d\right|^p +\frac{1}{2}\left| \bgd^{(b)}-G_d\right|^p + \left| \bgd-G_d\right|^p\right) .
\end{align*}

It follows from the first assertion of Lemma~\ref{lem:aux} that
\[ \EE\left[\left| \bgd^{(a)}-G_d\right|^p\right], \EE\left[\left| \bgd^{(b)}-G_d\right|^p\right], \EE\left[\left| \bgd-G_d\right|^p\right] = O(2^{-p\ell/2}), \]
for any $d\in D$. This leads to a bound on $\EE[|\Delta P_{\ell,2}|^p]$ of order $2^{-p\ell/2}$.\vspace{5pt}

\textbf{Bound on $\PP[A]+\PP[B]$.} It is straightforward from Assumption~\ref{assm:2} that $\PP[A]\leq c_0\epsilon$. Regarding a bound on $\PP[B]$, we have
\[ \PP[B] \leq \sum_{d\in D}\left( \PP[ |\bgd^{(a)}-G_d| \geq \frac{1}{2}c_2\epsilon]+\PP[ |\bgd^{(b)}-G_d| \geq \frac{1}{2}c_2\epsilon]+ \PP[ |\bgd-G_d| \geq \frac{1}{2}c_2\epsilon]\right). \]
Using the Markov inequality, we obtain
\begin{align*}
\PP\left[ |\bgd-G_d| \geq \frac{1}{2}c_2\epsilon\,|\, Y\right] & = \PP\left[ |\bgd-G_d|^p \geq \frac{1}{2^p}c_2^p\epsilon^p \,|\, Y\right] \\
& \leq \frac{2^p}{c_2^p\epsilon^p}\EE[ |\bgd-G_d|^p \,|\, Y] .
\end{align*}
Taking an outer expectation with respect to $Y$, the tower property provides
\begin{align*}
\PP\left[ |\bgd-G_d| \geq \frac{1}{2}c_2\epsilon\right] & = \EE_Y\left[\PP\left[ |\bgd-G_d| \geq \frac{1}{2}c_2\epsilon\,|\, Y\right]\right] \\
& \leq \frac{2^p}{c_2^p\epsilon^p}\EE[ |\bgd-G_d|^p ] .
\end{align*}
The first assertion of Lemma~\ref{lem:aux} states that the expectation on the right-most side above is of order $2^{-p\ell/2}$. Since similar bounds exist for $\PP[ |\bgd^{(a)}-G_d| \geq \frac{1}{2}c_2\epsilon]$ and $\PP[ |\bgd^{(b)}-G_d| \geq \frac{1}{2}c_2\epsilon]$, this proves a bound on $\PP[A]+\PP[B]$ of order
\[ \max(\epsilon, \epsilon^{-p}2^{-p\ell/2}) = 2^{-(p/(2p+2))\ell}.\]

\textbf{Bound on $\EE[|\Delta P_{\ell,2}|^2\bsone_{A^c\cap B^c}]$.} Finally let us consider the case where the event $A^c\cap B^c$ happens. Namely we have $\min_{y\in K}\| Y-y\| > \epsilon$ and for any $d\in D$
\[  |\bgd^{(a)}-G_d|,  |\bgd^{(b)}-G_d|,  |\bgd-G_d|< \frac{1}{2}c_2\epsilon. \]
For a particular value of $Y$, for any decision $d\neq d_{\opt}(Y)$, Assumption~\ref{assm:3} ensures
\begin{align*}
\overline{g_{d_\opt}}-\bgd & = (G_{d_\opt}-G_d)+(\overline{g_{d_\opt}}-G_{d_\opt})-(\bgd-G_d) \\
& > \min(c_1,c_2\epsilon) - \frac{1}{2}c_2\epsilon - \frac{1}{2}c_2\epsilon = \min(c_1-c_2\epsilon, 0).
\end{align*}
Therefore, for a sufficiently large $\ell$, we have $c_1-c_2\epsilon>0$ and so $\overline{g_{d_\opt}}-\overline{g_d}>0$. This means that $\arg\max_{d\in D}\overline{g_d}=d_\opt$, that is, we shall always choose an optimal decision, and so for $\overline{g_d}^{(a)}$ and $\overline{g_d}^{(b)}$. It follows that
\[ \Delta P_{\ell,2}= \frac{1}{2}\left( \max_{d\in D}\overline{g_d}^{(a)}+\max_{d\in D}\overline{g_d}^{(b)}\right) - \max_{d\in D}\overline{g_d} =\frac{1}{2}\left( \overline{g_{d_\opt}}^{(a)}+\overline{g_{d_\opt}}^{(b)}\right) - \overline{g_{d_\opt}} . \]
We see from the second assertion of Lemma~\ref{lem:aux} with $q=2$ that the expectation of the square of the right-most side above is of order $2^{-\min(p,4)\ell/2}$, which proves a bound on $\EE[|\Delta P_{\ell,2}|^2\bsone_{A^c\cap B^c}]$.\vspace{5pt}

Altogether, $\VV[\Delta P_\ell]$ is upper bounded as
\begin{align*}
 \VV[\Delta P_\ell] & \leq 2\EE[|\Delta P_{\ell,1}|^2]+2\EE[|\Delta P_{\ell,2}|^2] \\
 & = O(2^{-\min(p,4)\ell/2})+O(2^{-\ell}) \times O(2^{-(p-2)\ell/(2p+2)})+O(2^{-\min(p,4)\ell/2}) \\
 & = O(2^{-(3p/(2p+2))\ell}).
\end{align*}
Since the bound on $\EE[|\Delta P_\ell|]$ can be proven in a similar way, we omit the proof.

\subsection{Auxiliary results}\label{subsec:aux}

Before stating necessary auxiliary results, let us first recall the following fact. The proof can be found in \cite[Lemma~1]{GG19}.
\begin{lemma}\label{lem:deviation}
Let $X$ be a real-valued random variable with mean zero and $\overline{X}_N$ be an average of $N$ i.i.d.\ samples of $X$. If $\EE[|X|^p]<\infty$ for $p\geq 2$, there exists a constant $C_p>0$ depending only on $p$ such that
\[ \EE[|\overline{X}_N|^p] \leq C_p\frac{\EE[|X|^p]}{N^{p/2}}\quad \text{and}\quad \PP[|\overline{X}_N|>c]\leq C_p\frac{\EE[|X|^p]}{c^pN^{p/2}}. \]
\end{lemma}

\begin{lemma}\label{lem:aux}
If Assumptions~\ref{assm:1} and \ref{assm:4} are satisfied, then for any $d\in D$ we have
\[ \EE\left[ |\bgd - G_d|^q \right] = O(2^{-q\ell/2}) \quad \text{for any $1\leq q\leq p$} \]
and
\[ \EE\left[ \left| \frac{1}{2} \left(\bgd^{(a)} + \bgd^{(b)}\right) - \bgd \right|^q \right] = O(2^{-\min(p,2q)\ell/2}) \quad \text{for any $1\leq q\leq p$}, \]
where $p$ is as defined in Assumption~\ref{assm:4}.
\end{lemma}

\begin{proof}
Note that because of the fact that the same i.i.d.\ samples of $\theta$ are used both in the denominator and numerator of $\bgd^{(a)},\bgd^{(b)},\bgd$, respectively, it follows from Assumption~\ref{assm:1} that $|\bgd^{(a)}|,|\bgd^{(b)}|,|\bgd| \leq F_{max}$. Also we have
\[ |F_d(Y)| \leq F_{max}\EE_{\theta}\left[\rho(Y\,|\, \theta)\right] = F_{max}\rho(Y). \]

For the first assertion, we give a proof only for $q\!=\!p$, since the result for $q<p$ follows immediately
from H\"{o}lder's inequality.

We define an extreme event $E$ as
\[ E: \left|\frac{\brhoY}{\rho(Y)} - 1\right| > \frac{1}{2} \]
and then
\[ \EE[|\bgd - G_d|^p] = \EE[|\bgd - G_d|^p\bsone_E ] + \EE[|\bgd - G_d|^p\bsone_{E^c} ]. \]

For the first term, due to H\"{o}lder's inequality, tower property of expectation and Lemma~\ref{lem:deviation}, we have
\begin{align}
\EE[|\bgd - G_d|^p\bsone_E ] & \leq (2F_{max})^p \PP[E] \notag \\
&\leq 2^{2p-p\ell/2} M_0^{-p/2}C_p F_{max}^p \EE_Y \left[ \EE_\theta\left[ \left|\frac{\rho(Y\,|\, \theta)}{\rho(Y)}-1\right|^p \right] \right].
\label{eq:term1}
\end{align}

For the second term, when $\left| \brhoY / \rho(Y) - 1 \right| \leq 1/2$ we have
\begin{align*}
|\bgd - G_d|^p & \leq 2^{p-1} \left| \frac{\bfrhod}{\brhoY} - \frac{F_d(Y)}{\brhoY} \right|^p + 2^{p-1} \left| \frac{F_d(Y)}{\brhoY} - \frac{F_d(Y)}{\rho(Y)} \right|^p \\
& \leq \frac{2^{2p-1}}{\rho(Y)^p} \left|\bfrhod - F_d(Y)\right|^p + 2^{2p-1} F_{max}^p \left|\frac{\brhoY}{\rho(Y)} - 1\right|^p
\end{align*}
and therefore, again by Lemma 1,
\begin{align}
\EE[ |\bgd - G_d|^p\bsone_{E^c} ] & \leq 2^{2p-1-p\ell/2} M_0^{-p/2}C_p \EE_Y\left[ \EE_\theta\left[ \left| \frac{f_d(\theta)\rho(Y\,|\, \theta)-F_d(Y)}{\rho(Y)}\right|^p\right] \right] \notag \\
 &\quad +  2^{2p-1-p\ell/2}M_0^{-p/2}C_p F_{max}^p \EE_Y \left[ \EE_\theta\left[ \left|\frac{\rho(Y\,|\, \theta)}{\rho(Y)}-1\right|^p \right] \right].
\label{eq:term2}
\end{align}

Note that
\[ \left|\frac{\rho(Y\,|\, \theta)}{\rho(Y)}-1\right|^p \leq \max\left( \frac{\rho(Y\,|\, \theta)}{\rho(Y)}, 1\right)^p \leq \left(\frac{\rho(Y\,|\, \theta)}{\rho(Y)}\right)^p + 1, \]
which gives
\begin{align}
 \EE_Y \left[ \EE_\theta\left[\left|\frac{\rho(Y\,|\, \theta)}{\rho(Y)}-1\right|^p\right]\right] \leq \EE_Y \left[ \EE_\theta\left[ \left(\frac{\rho(Y\,|\, \theta)}{\rho(Y)}\right)^p\right]\right] + 1.
\label{eq:mom1}
\end{align}
Similarly,
\begin{align*}
|f_d(\theta)\rho(Y\,|\, \theta)-F_d(Y)|^p & \leq  2^{p-1} \left( |f_d(\theta)\rho(Y\,|\, \theta)|^p + |F_d(Y)|^p \right) \\
& \leq 2^{p-1}F_{max}^p (\rho(Y\,|\, \theta)^p +  \rho(Y)^p),
\end{align*}
which gives
\begin{align}
\EE_Y\left[ \EE_\theta\left[ \left| \frac{f_d(\theta)\rho(Y\,|\, \theta)-F_d(Y)}{\rho(Y)}\right|^p\right] \right] \leq 2^{p-1}F_{max}^p \left( \EE_Y \left[ \EE_\theta\left[ \left(\frac{\rho(Y\,|\, \theta)}{\rho(Y)}\right)^p\right]\right] + 1\right).
\label{eq:mom2}
\end{align}
Combining the bounds in \eqref{eq:term1} and \eqref{eq:term2}, and then using \eqref{eq:mom1} and \eqref{eq:mom2}, we get
\[ \EE[\, |\bgd - G_d|^p ] \leq (3\cdot 2^{2p-1}\!+\!2^{3p-2}) 2^{-p\ell/2}M_0^{-p/2}C_p F_{max}^p \left( \EE_Y \left[ \EE_\theta\left[ \left(\frac{\rho(Y\,|\, \theta)}{\rho(Y)}\right)^p\right]\right] + 1\right), \]
which completes the proof of the first assertion.

For the second assertion we define the extreme event $E$ as
\[ E: \max\left( \left|\frac{\brhoY^{(a)}}{\rho(Y)} -1\right|, \left|\frac{\brhoY^{(b)}}{\rho(Y)} -1\right| \right) > \frac{1}{2} \]
and then
\begin{align*}
  \EE\left[\left| \frac{1}{2} \left(\bgd^{(a)} + \bgd^{(b)}\right) - \bgd \right|^q \right] & = \EE\left[\left| \frac{1}{2} \left(\bgd^{(a)} + \bgd^{(b)}\right) - \bgd \right|^q \bsone_E \right] \\
 &\quad + \EE\left[\left| \frac{1}{2}\left(\bgd^{(a)} + \bgd^{(b)}\right) - \bgd \right|^q \bsone_{E^c} \right]
\end{align*}
For the first term, following similar reasoning to before, we have
\begin{align}
& \EE\left[\left| \frac{1}{2} \left(\bgd^{(a)} + \bgd^{(b)}\right) - \bgd \right|^q \bsone_E \right] \notag \\
& \leq (2F_{max})^q \left( \PP\left[ \left|\frac{\brhoY^{(a)}}{\rho(Y)} -1\right| > \fracs{1}{2}\right] + \PP\left[ \left|\frac{\brhoY^{(b)}}{\rho(Y)} -1\right| > \fracs{1}{2}\right] \right)\notag \\
& \leq 2^{p+q+1} C_p 2^{-p(\ell-1)/2}M_0^{-p/2} F_{max}^q \left( \EE_Y \left[ \EE_\theta\left[ \left(\frac{\rho(Y\,|\, \theta)}{\rho(Y)}\right)^p\right]\right] + 1\right).
\label{eq:term3}
\end{align}
For the second term, we use the antithetic properties \eqref{eq:anti1} and \eqref{eq:anti2} to give
\begin{align*}
\frac{1}{2} \left(\bgd^{(a)} + \bgd^{(b)}\right) - \bgd & = \frac{1}{2} \left( \bfrhod^{(a)} - F_d(Y)\right)
    \left(\frac{1}{\brhoY^{(a)}}-\frac{1}{\rho(Y)}\right)\\
&\quad + \frac{1}{2} \left( \bfrhod^{(b)} - F_d(Y)\right)\left(\frac{1}{\brhoY^{(b)}}-\frac{1}{\rho(Y)}\right)\\
&\quad - \left( \bfrhod - F_d(Y)\right)\left(\frac{1}{\brhoY}-\frac{1}{\rho(Y)}\right)\\
&\quad + \frac{1}{2} \frac{F_d(Y)}{\rho(Y)} \left(\frac{\rho(Y)}{\brhoY^{(a)}}- 2 + \frac{\brhoY^{(a)}}{\rho(Y)} \right)\\
&\quad + \frac{1}{2} \frac{F_d(Y)}{\rho(Y)} \left(\frac{\rho(Y)}{\brhoY^{(b)}}- 2 + \frac{\brhoY^{(b)}}{\rho(Y)} \right)\\
&\quad - \frac{F_d(Y)}{\rho(Y)} \left(\frac{\rho(Y)}{\brhoY}- 2 + \frac{\brhoY}{\rho(Y)} \right).
\end{align*}
If $|\brhoY^{(a)} / \rho(Y) -1| < 1/2$ and $|\brhoY^{(b)} / \rho(Y) -1| < 1/2$, it follows that $|\brhoY / \rho(Y) -1| < 1/2$, and therefore by applying Jensen's inequality we obtain
\begin{align*}
\left|\frac{1}{2} \left(\bgd^{(a)} + \bgd^{(b)}\right) - \bgd\right|^q & \leq 2^{2q-3}\left| \bfrhod^{(a)} - F_d(Y)\right|^q \left|\frac{1}{\brhoY^{(a)}}-\frac{1}{\rho(Y)}\right|^q \\
&\quad + 2^{2q-3}\left| \bfrhod^{(b)} - F_d(Y)\right|^q\left|\frac{1}{\brhoY^{(b)}}-\frac{1}{\rho(Y)}\right|^q\\
&\quad + 2^{2q-2}\left| \bfrhod - F_d(Y)\right|^q\left|\frac{1}{\brhoY}-\frac{1}{\rho(Y)}\right|^q \\
&\quad + 2^{2q-3}\frac{|F_d(Y)|^q}{\rho(Y)^q} \left|\frac{\rho(Y)}{\brhoY^{(a)}}- 2 + \frac{\brhoY^{(a)}}{\rho(Y)} \right|^q\\
&\quad + 2^{2q-3}\frac{|F_d(Y)|^q}{\rho(Y)^q} \left|\frac{\rho(Y)}{\brhoY^{(b)}}- 2 + \frac{\brhoY^{(b)}}{\rho(Y)} \right|^q\\
&\quad + 2^{2q-2}\frac{|F_d(Y)|^q}{\rho(Y)^q} \left|\frac{\rho(Y)}{\brhoY}- 2 + \frac{\brhoY}{\rho(Y)} \right|^q \\
& \leq 2^{3q-3}\left| \bfrhod^{(a)} - F_d(Y)\right|^q \left|\frac{\brhoY^{(a)}}{\rho(Y)}-1\right|^q\frac{1}{\rho(Y)^q} \\
&\quad + 2^{3q-3}\left| \bfrhod^{(b)} - F_d(Y)\right|^q\left|\frac{\brhoY^{(b)}}{\rho(Y)}-1\right|^q\frac{1}{\rho(Y)^q}\\
&\quad + 2^{3q-2}\left| \bfrhod - F_d(Y)\right|^q\left|\frac{\brhoY}{\rho(Y)}-1\right|^q\frac{1}{\rho(Y)^q} \\
&\quad + 2^{3q-3}|F_d(Y)|^q \left|\frac{\brhoY^{(a)}}{\rho(Y)}-1\right|^{2q}\frac{1}{\rho(Y)^q}\\
&\quad + 2^{3q-3}|F_d(Y)|^q \left|\frac{\brhoY^{(b)}}{\rho(Y)}-1\right|^{2q}\frac{1}{\rho(Y)^q}\\
&\quad + 2^{3q-2}|F_d(Y)|^q \left|\frac{\brhoY}{\rho(Y)}-1\right|^{2q}\frac{1}{\rho(Y)^q}.
\end{align*}

Looking at the third term, if $p \geq 2q$ then H\"{o}lder's inequality together with Lemma~\ref{lem:deviation}, and \eqref{eq:mom1} and \eqref{eq:mom2} gives
\begin{align*}
& \EE \left[ \left| \bfrhod - F_d(Y)\right|^q\left|\frac{\brhoY}{\rho(Y)}-1\right|^q\frac{1}{\rho(Y)^q} \right] \\
&\leq \left( \EE \left[ \left| \bfrhod - F_d(Y)\right|^{2q} /\rho(Y)^{2q} \right] \EE \left[ \left|\frac{\brhoY}{\rho(Y)}-1\right|^{2q} \right]\right)^{1/2} \\
&\leq 2^{q-1/2} C_{2q} 2^{-q\ell}M_0^{-q} F_{max}^q \left( \EE_Y \left[ \EE_\theta\left[ \left(\frac{\rho(Y\,|\, \theta)}{\rho(Y)}\right)^{2q}\right]\right] + 1\right).
\end{align*}
If $2\leq p<2q$ then we need to modify the argument slightly, using H\"{o}lder's inequality when $q<p$ and a simple identity when $q=p$, to obtain
\begin{align*}
& \EE \left[ \left| \bfrhod - F_d(Y)\right|^q\left|\frac{\brhoY}{\rho(Y)}-1\right|^q\frac{1}{\rho(Y)^q} \bsone_{E^c}\right] \\
& \leq 2^{p-2q}\EE \left[ \left| \bfrhod - F_d(Y)\right|^q\left|\frac{\brhoY}{\rho(Y)}-1\right|^{p-q}\frac{1}{\rho(Y)^q} \right] \\
& \leq 2^{p-2q} \left( \EE \left[ \left| \bfrhod - F_d(Y)\right|^p /\rho(Y)^p \right]\right)^{q/p}\left( \EE \left[ \left|\frac{\brhoY}{\rho(Y)}-1\right|^p \right]\right)^{1-q/p} \\
&\leq 2^{p-q-q/p} C_p 2^{-p\ell/2}M_0^{-p/2} F_{max}^q \left( \EE_Y \left[ \EE_\theta\left[ \left(\frac{\rho(Y\,|\, \theta)}{\rho(Y)}\right)^p\right]\right] + 1\right).
\end{align*}

Looking at the sixth term, by Lemma 1 and \eqref{eq:mom1} we obtain
\begin{align*}
& \EE \left[ |F_d(Y)|^q \left|\frac{\brhoY}{\rho(Y)}-1\right|^{2q}\frac{1}{\rho(Y)^q}\right]\\
&\leq  C_{2q} 2^{-q\ell}M_0^{-q} F_{max}^q \left( \EE_Y \left[ \EE_\theta\left[ \left(\frac{\rho(Y\,|\, \theta)}{\rho(Y)}\right)^{2q}\right]\right] + 1\right),
\end{align*}
when $p\geq 2q$, while for $2\leq p< 2q$ we get
\begin{align*}
& \EE \left[ |F_d(Y)|^q \left|\frac{\brhoY}{\rho(Y)}-1\right|^{2q}\frac{1}{\rho(Y)^q} \bsone_{E^c}\right] \\
&\leq  2^{p-2q}C_p 2^{-p\ell/2}M_0^{-p/2} F_{max}^q \left( \EE_Y \left[ \EE_\theta\left[ \left(\frac{\rho(Y\,|\, \theta)}{\rho(Y)}\right)^p\right]\right] + 1\right).
\end{align*}

The other terms can be treated similarly and therefore, adding in the bound from \eqref{eq:term3}, we determine that there exists a constant $C$, dependent only on $p$ and $q$, such that
\begin{align*}
& \EE\left[\left|\frac{1}{2} \left(\bgd^{(a)} + \bgd^{(b)}\right) - \bgd\right|^q \right] \\
&\leq C 2^{-\min(p,2q)\ell/2}M_0^{-\min(p,2q)/2} F_{max}^{q} \left( \EE_Y \left[ \EE_\theta\left[ \left(\frac{\rho(Y\,|\, \theta)}{\rho(Y)}\right)^{\min(p,2q)}\right]\right] + 1\right),
\end{align*}
completing the proof of the second assertion.
\end{proof}

\begin{remark}\label{rem:gmax}
It follows from Assumption~\ref{assm:1} that
\[\left|\max_{d\in D}f_d(\theta)\right| \leq F_{max}\]
for any $\theta$. Given that the antithetic property \eqref{eq:anti3} holds, the same proof strategy leads to
\[ \EE[|\Delta P_{\ell,1}|^q]=\EE\left[ \left| \overline{g}_{\max}-\frac{1}{2}\left(\overline{g}_{\max}^{(a)}+\overline{g}_{\max}^{(b)}\right) \right|^q \right] = O(2^{-\min(p,2q)\ell/2}), \]
where $p$ is as defined in Assumption~\ref{assm:4}.
\end{remark}

\begin{remark}\label{rem:non-antithetic}
Without the antithetic technique, the best that can be deduced from Lemma~\ref{lem:aux} is that
\begin{align*}
 \EE\left[ |\bgd^{(a)}-\bgd|^q \right] & = \EE\left[ |\bgd^{(a)}-G_d-\bgd+G_d|^q \right] \\
 & \leq 2^{q-1}\EE\left[ |\bgd^{(a)}-G_d|^q \right]+2^{q-1}\EE\left[ |\bgd-G_d|^q \right] = O(2^{-q\ell/2})
\end{align*}
for $1\leq q\leq p$. This leads to the variance of the non-antithetic correction $\Delta P_\ell = P_\ell-P_{\ell-1}^{(a)}$ being $O(2^{-\ell})$. Since we now have $\beta=1$ in Theorem~\ref{thm:mlmc}, the non-antithetic MLMC estimator is in the second regime regardless of the value of $p\geq 2$. The resulting total computational complexity is thus $O(\varepsilon^{-2}(\log \varepsilon^{-1})^2)$, which is not optimal.
\end{remark}

\section{Numerical experiments}\label{sec:numer}

\subsection{Problem setting}
To demonstrate the efficiency of our MLMC estimator, we conduct some numerical experiments using a modified version of the model originally introduced in \cite{ALC04} and explored further in \cite{SOBB15}. The original version of this artificial cost-effectiveness model compares a new treatment to some standard of care on the prevention of a critical event, denoted by $E$. We extend this to three decision options $D=\{1,2,3\}$. We also modify the model to include correlations between odds ratios (OR) of the critical events ($OR_{E,2}$, $OR_{E,3}$), treatment costs ($C_{T,2}$, $C_{T,3}$), and probabilities of side effects ($P_{SE,2}$, $P_{SE,3}$). The resulting vector $\theta$ consists of 12 random variables $L$, $Q_E$, $Q_{SE}$, $C_E$, $C_{SE}$, $C_{T,2}$, $C_{T,3}$, $P_{E,1}$, $OR_{E,2}$, $OR_{E,3}$, $P_{SE,2}$, $P_{SE,3}$. The model contains only three constants $C_{T,1}, P_{SE,1}, \lambda$, while two parameters $P_{E,2}$ and $P_{E,3}$ are defined as functions of $P_{E,1}, OR_{E,2}$ and $P_{E,1}, OR_{E,3}$, respectively. We refer to Table~\ref{table:ades2004} for a detailed description on these model inputs. For each treatment $d\in D$, the net benefit is defined by
\begin{align*}
	f_d(\theta) &= P_{SE,d}P_{E,d}\left[\lambda\left(L\frac{1+Q_E}{2}-Q_{SE}\right)-(C_{SE}+C_E)\right] \\
		& \quad + P_{SE,d}(1-P_{E,d})\left[\lambda(L-Q_{SE})-C_{SE}\right] \\
		& \quad + (1-P_{SE,d})P_{E,d}\left[\lambda L\frac{1+Q_E}{2}-C_E\right] \\
		& \quad + (1-P_{SE,d})(1-P_{E,d})\lambda L-C_{T,d}.
\end{align*}
This net benefit, as in \cite{ALC04}, is multilinear, with the consequence that expected values of random variables can be plugged in to form estimates of the EVSI. However, the additional decision options, correlations, and larger number of parameters make the problem more realistic, and computationally challenging, than the original version. Repeating the standard Monte Carlo estimation with $10^8$ i.i.d.\ samples of $\theta$ 10 times, EVPI is estimated as 4,063.5 with the standard error equal to 0.66.
\begin{table}[t]
	\caption{Model inputs involved in our experiments. Note that $\text{log-normal}(\mu,\Sigma)$ and $\text{logit-normal}(\mu,\Sigma)$ denote the log-normal and logit-normal distributions, respectively, with $\mu$ and $\Sigma$ being the mean vector and the covariance matrix of the corresponding normal distribution, respectively. The logit function $\text{logit}(p)$ is the transformation mapping probability $p$ to the logarithm of the odds $\log(p/(1-p))$.}\label{table:ades2004}
	\small{ \begin{center} 
	\begin{tabular}{|l|l|l|}
		\hline
		Description & Parameter & Distribution \\
		\hline \hline
		Lifetime remaining & $L$ & $N(30, 25)$ \\ \hline
		QALY after critical event, per year & $Q_E$ & $\text{logit-normal}\left(0.6, 1/36\right)$ \\ \hline
		QALY decrement due to side effects & $Q_{SE}$ & $N(0.7, 0.01)$ \\ \hline
		Cost of critical event & $C_E$ & $N(2\times 10^5, 10^8)$ \\ \hline
		Cost of side effect & $C_{SE}$ & $N(10^5, 10^8)$ \\ \hline
		Cost of treatment $d=1$ & $C_{T,1}$ & 0 (constant) \\ \hline
		Cost of treatments $d=2,3$ & $C_{T,d}$ & \multirow{2}{*}{$N\left(\begin{pmatrix} 1.5\times 10^4 \\ 2\times 10^4\end{pmatrix}, \begin{pmatrix}300 & 100 \\ 100 & 500\end{pmatrix}\right)$} \\
		& & \\ \hline
		Probability of critical event & $P_{E,1}$ & $\text{Beta}(15, 85)$ \\
		on treatment $d=1$ & &  \\ \hline
		Odds ratios of critical event & $OR_{E,d}$ & \multirow{2}{*}{$\text{log-normal}\left(\begin{pmatrix} -1.5 \\ -1.75\end{pmatrix}, \begin{pmatrix}0.11 & 0.02 \\ 0.02 & 0.06\end{pmatrix}\right)$} \\
		relative to treatment $d=1$ & &  \\
		$(P_{E,d}/(1-P_{E,d}))/(P_{E,1}/(1-P_{E,1}))$ & & \\ \hline
		Probability of critical event & $P_{E,d}$ & Derived from\\
		on treatments $d=2,3$ & & $P_{E,1}$ and $OR_{E,d}$\\ \hline
		Probability of side effect & $P_{SE,1}$ & 0 (constant) \\
		on treatment $d=1$ & & \\ \hline
		Probability of side effect & $P_{SE,d}$ & \multirow{2}{*}{$\text{logit-normal}\left(\begin{pmatrix} -1.4 \\ -1.1\end{pmatrix}, \begin{pmatrix}0.10 & 0.05 \\ 0.05 & 0.25\end{pmatrix}\right)$} \\
		 on treatments $d=2,3$ & & \\ \hline
		Monetary value of 1 QALY & $\lambda$ & \$75,000 (constant) \\
		\hline
	\end{tabular} \end{center} }
\end{table}

Our final modification to improve practical relevance is to assume an annual population who experience this disease as 2,500, a technology horizon of 10 years, and an annual discount factor of 1.035, giving a total discounted population that will benefit from sampling as 21,519. Estimated per-person EVSI is multiplied by this factor to get a population EVSI which can be compared with experiment costs.

We consider three hypothetical experiment designs for which we want to estimate the EVSI. Unlike \cite{ALC04,SOBB15}, all experiments inform more than one parameter and we consider cases where the parameters being studied (subset of $\theta$ informed by sample information $Y$) are correlated with those not being studied (subset of $\theta$ not informed by $Y$). These modifications again improve the practical relevance of the example.
\begin{enumerate}
\item An observational study on side effects in a cohort of $n_p=100$ patients on treatment $d=2$, which informs three parameters $Q_{SE}$, $C_{SE}$ and $P_{SE,2}$. The sample information $Y$ in this scenario, denoted by Scenario~1, is defined as the three-dimensional vector $(Y_1,Y_2,Y_3)$, where $Y_1\sim N(Q_{SE}, 4/n_p)$, $Y_2\sim N(C_{SE}, 10^4/n_p)$ and $Y_3\sim \mathrm{Binomial}(n_p, P_{SE,2})$. The hypothetical study costs are \$50,000 for set-up and \$250 for each patient giving a total of \$75,000.
\item A small two-arm randomized controlled trial (RCT) comparing treatments $d=1$ and $d=3$ with $n_p=100$ patients, which informs three parameters $OR_{E,3}$, $C_{T,3}$ and $P_{SE,3}$. The sample information $Y$ in this scenario, denoted by Scenario~2, is defined as the three-dimensional vector $(Y_1,Y_2,Y_3)$, where $Y_1\sim N(\log(OR_{E,3}), 4/n_p)$, $Y_2\sim N(C_{T,3}, 10^4/n_p)$ and $Y_3\sim \mathrm{Binomial}(n_p, P_{SE,3})$. This hypothetical RCT costs \$200,000 for set-up plus \$1000 for each randomized patient, giving a total of \$400,000.
\item A larger multicentre two-arm randomized controlled trial comparing treatments $d=1$ and $d=3$ with $n_p=1000$ patients, which informs six parameters $P_{E,1}$, $OR_{E,3}$, $C_{T,3}$, $P_{SE,3}$, $C_{SE}$ and $C_E$. The sample information $Y$ in this scenario, denoted by Scenario~3, is defined as the six-dimensional vector $(Y_1,\ldots,Y_6)$, where $Y_1\sim \mathrm{Binomial}(n_p, P_{E,1})$, $Y_2\sim N(\log(OR_{E,3}), 4/n_p)$, $Y_3\sim N(C_{T,3}, 10^4/n_p)$, $Y_4\sim \mathrm{Binomial}(n_p, P_{SE,3})$, $Y_5\sim N(C_{SE}, 10^4/n_p)$ and $Y_6\sim N(C_{E}, 10^4/n_p)$. This hypothetical multicentre RCT costs \$200,000 set-up, \$1500 for each randomized patient, and \$50,000 for each additional 100 patients added to the study (owing to the cost of establishing a new centre). The total cost is \$4,200,000.
\end{enumerate}

\subsection{Results and discussion}
Throughout our experiments below, we always set $M_0=16$. Here we refer to \cite[p.~516]{GH19} for the discussion on how to choose $M_0$ to minimise the total complexity of the MLMC estimator. Except for $P_{SE,2}$ for Scenario~1 and $P_{SE,3}$ for Scenarios 2 \& 3, the exact posterior distributions of the relevant random input parameters given $Y$ are available and used as importance distributions. Regarding $P_{SE,2}$ for Scenario~1, we approximate the prior distribution by the beta distribution with the same mean and variance, for which the exact posterior distribution given $Y_3$ is available and can be used as an importance distribution. We employ the same approach for $P_{SE,3}$ for Scenarios 2 \& 3.

Following the MLMC implementation by Giles \cite[Section~3]{G15}, we first conduct convergence tests on a sequence of the correction variables $\Delta P_\ell$ and then estimate $\EVPI-\EVSI$ for various values of the root-mean-square accuracy $\varepsilon$. The required computational costs for our MLMC estimator are compared with those for the standard nested Monte Carlo estimator of the form
\[ \frac{1}{N}\sum_{n=1}^{N}P_L^{(n)}, \]
for the same maximum level $L$ with the MLMC estimator. Here we note that this nested Monte Carlo estimator is different from that given in \eqref{eq:nmc}. We also note that function approximation techniques \cite{SOBB15,HMB16,JA18} may be applied to this example. However, performance comparison with the nested MCMC estimator or function approximation-based estimators is beyond the scope of this article and will be addressed in future work.

The results for Scenario~1 are shown in Figure~\ref{fig:result1}. In the left top panel, the variances of both $P_\ell$ and $\Delta P_\ell$ are plotted as functions of the level $\ell$. Here the variances are estimated empirically by using $10^5$ i.i.d.\ samples at each level. While the variance of $P_\ell$ takes an almost constant value, the variance of $\Delta P_\ell$ decreases geometrically as the level increases. The linear regression of the data for the range $2\leq \ell\leq 8$ provides an estimation of $\beta$ as 1.62, which is slightly better than what we expect from the theoretical analysis. Using the same i.i.d.\ samples of $P_\ell$ and $\Delta P_\ell$, we also estimate their absolute mean values, which are plotted as functions of $\ell$ in the right top panel. The absolute mean value of $P_\ell$ takes an almost constant value, whereas the absolute mean value of $\Delta P_\ell$ decreases geometrically as the level increases. The linear regression of the data for the range $2\leq \ell\leq 8$ provides an estimation of $\alpha$ as 1.16, which again is slightly better than our theoretical result. These convergence results confirm that our MLMC estimator is in the first regime of Theorem~\ref{thm:mlmc} for this scenario.
\begin{figure}[t]
\centering
\includegraphics[width=0.7\textwidth]{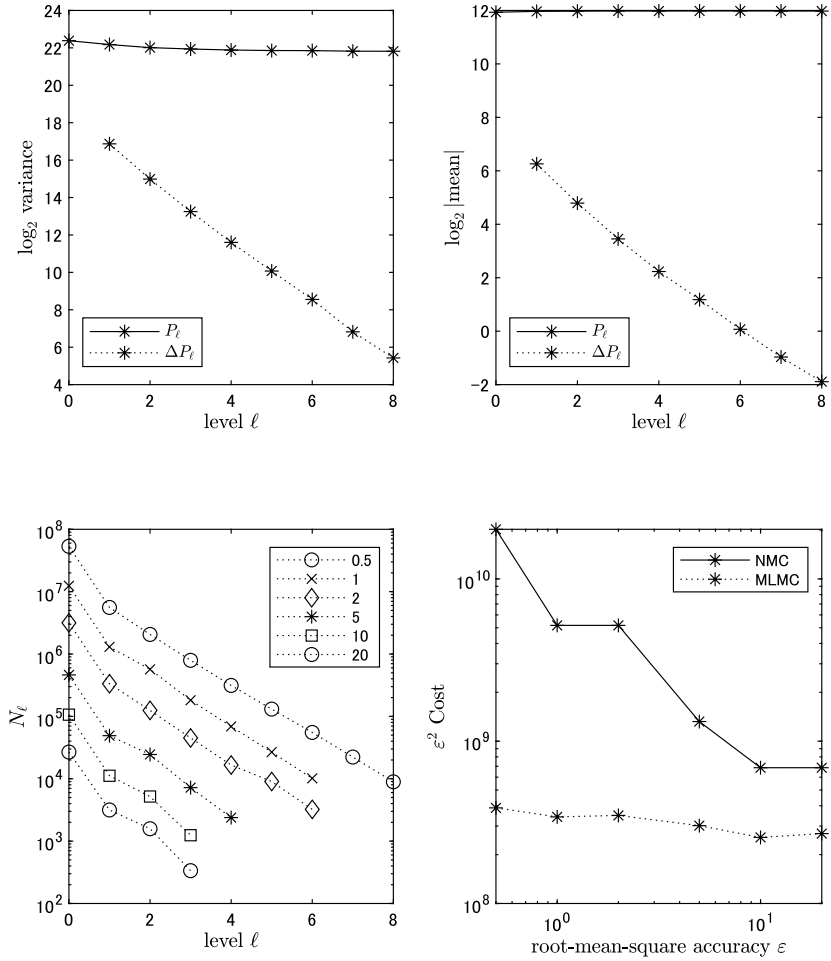}
\caption{MLMC results for Scenario~1}
\label{fig:result1}
\end{figure}

For a given root-mean-square accuracy $\varepsilon$, the necessary maximum level $L$ and the corresponding numbers of samples $N_\ell$ for $\ell=0,\ldots,L$ are estimated by using \cite[Algorithm~1]{G15}.  Each line in the left bottom panel of Figure~\ref{fig:result1}  shows the values of $N_\ell$ for a particular value of $\varepsilon\in \{0.5, 1, 2, 5, 10, 20\}$. It can be seen that the required maximum level $L$ increases as $\varepsilon$ decreases so that the bias becomes small enough. For each fixed $\varepsilon$, the number of samples $N_\ell$ deceases geometrically with the level, and most of the samples are allocated on the lower levels. In fact, the optimal allocation of $N_\ell$ is given by $N_\ell\propto \varepsilon^{-2}2^{-(\beta+\gamma)\ell/2}$ \cite{G15}, and our experimental results agree quite well with this. 

It should be noted that the values of both $\alpha$ and $\beta$ are estimated on-the-fly as the computation is performed, with the value of $\alpha$ being used to determine when the bias $\EE[P-P_L]$ has converged sufficiently as $L$ increases. Further details on the MLMC software implementation are available in \cite{G15}.

The total computational cost for the MLMC estimator is given by the sum $\sum_{\ell=0}^{L}2^\ell N_\ell=:C_{\MLMC}$ and it is expected from Theorem~\ref{thm:mlmc} and the above convergence results that $\varepsilon^2C_{\MLMC}$ is independent of $\varepsilon$. On the other hand, the total computational cost for the nested Monte Carlo estimator, denoted by $C_{\NMC}$, is of order $\varepsilon^{-2-1/\alpha}$ theoretically, so that $\varepsilon^2C_{\NMC}$ should be proportional to $\varepsilon^{-1/\alpha}$ and depends strongly on $\varepsilon$.  Based on this observation, the values of $\varepsilon^2C_{\MLMC}$ and $\varepsilon^2C_{\NMC}$ are plotted as functions of $\varepsilon$ with dashed and solid lines, respectively, in the right bottom panel. As expected, $\varepsilon^2C_{\NMC}$ increases significantly as $\varepsilon$ decreases, while $\varepsilon^2C_{\MLMC}$ takes an almost constant value. As a consequence, the superiority of the MLMC estimator becomes more remarkable for smaller $\varepsilon$. For the smallest value $\varepsilon=0.5$ in this experiment, the MLMC estimator achieves computational saving of factor 52. This factor is expected to get even larger if we set a smaller value for $\varepsilon$.

We obtain similar results for Scenarios~2 \& 3 as shown in Figures~\ref{fig:result2} and \ref{fig:result3}, respectively. The estimated values of $\alpha$ and $\beta$ are 0.99 and 1.48, respectively, for both these scenarios, which agrees quite well with our theoretical finding. The MLMC estimator achieves computational saving of factors 35 and 118 for the smallest value $\varepsilon=0.5$. The estimates of the difference $\EVPI-\EVSI$ are 4,039, 3,033 and 2,277 for three scenarios, respectively, with the root-mean-square accuracy equal to 0.5. Subtracting these values from the estimate of EVPI, which is 4,064, the per-person EVSI are estimated as 25, 1,031 and 1,787, respectively. Although we round the estimates to the nearest integers here for simplicity, each EVSI is estimated with the root-mean-square accuracy at most $0.66+0.5=1.06$. As comparison, by using the MLMC estimator for EVPPI \cite{GG19}, the per-person EVPPI for $(Q_{SE}, C_{SE}, P_{SE,2})$, $(OR_{E,3}, C_{T,3}, P_{SE,3})$ and $(P_{SE,1}, OR_{E,3}, C_{T,3}, P_{SE,3}, C_{SE}, C_{E})$ are estimated as 86, 1,400 and 1,914, each of which is certainly larger than the corresponding EVSI, confirming face validity of our EVSI estimates.

Multiplying the EVSI estimates by the discounted population that will benefit from sampling gives the population EVSI of \$537,975‬, \$22,186,089, and \$38,454,453 (with the root-mean-square accuracy at most $1.06\times 21,519=22,810$) for Scenarios~1, 2, and 3, respectively, which are all larger than their hypothetical costs and are thus cost-effective. Subtracting the study costs gives the expected net benefit of sampling (ENBS) \cite{WSCAAbook}. In our three scenarios, The ENBS values are estimated as \$462,975, \$21,786,089, and \$34,254,453, respectively. The ENBS for the third experiment, the large multicenter randomized controlled trial, is the highest and is thus the study that should be funded.

\begin{figure}[t]
\centering
\includegraphics[width=0.7\textwidth]{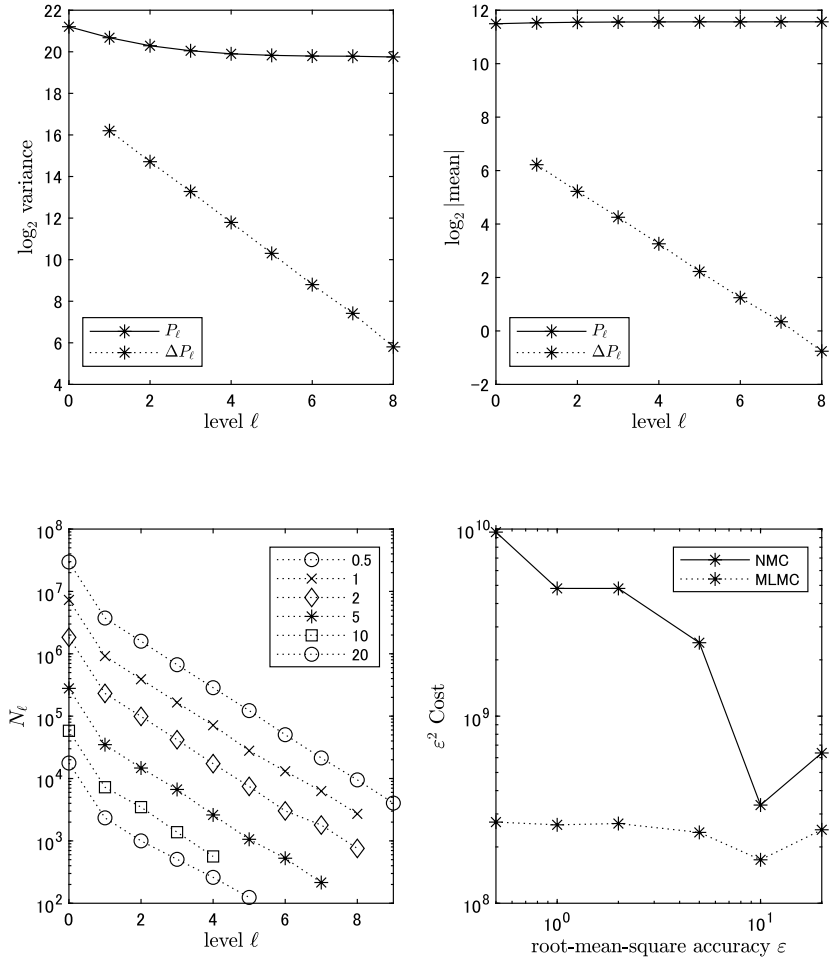}
\caption{MLMC results for Scenario~2}
\label{fig:result2}
\end{figure}
\begin{figure}[t]
\centering
\includegraphics[width=0.7\textwidth]{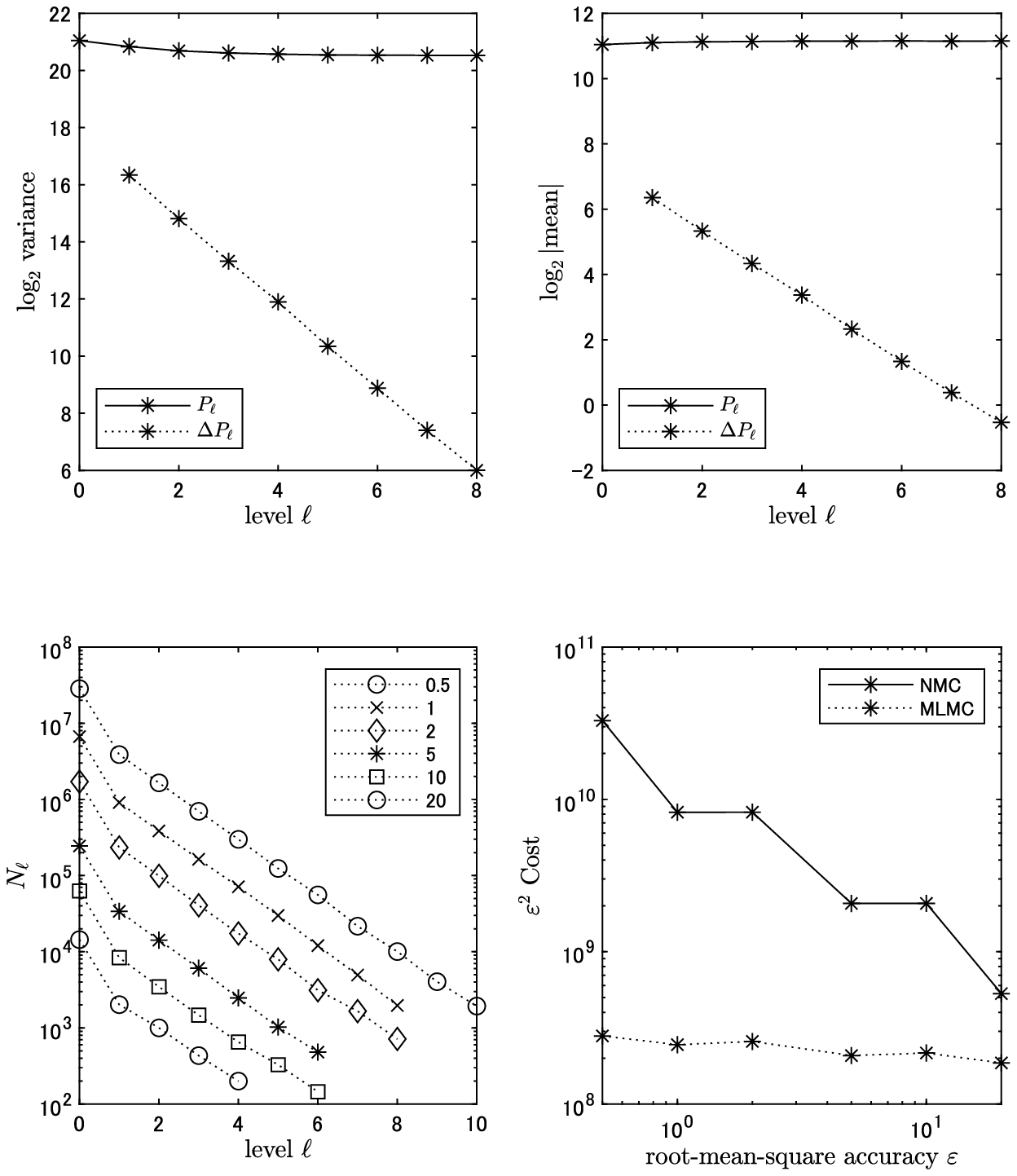}
\caption{MLMC results for Scenario~3}
\label{fig:result3}
\end{figure}

\section{Conclusions}\label{sec:concl}
In this paper, we have developed a multilevel Monte Carlo estimator for EVSI, the expected value of sample information. The key difference from the multilevel estimation for EVPPI, as studied in \cite{GG19}, is to use Bayes' theorem directly to rewrite the expectation with respect to the posterior distribution of input random variables $\theta$ given an observation $Y$ into the ratio of the expectations with respect to the prior distribution of $\theta$, and then to estimate each of the expectations by using the same i.i.d.\ samples of $\theta$. As shown in Lemma~\ref{lem:aux}, we prove under Assumptions~\ref{assm:1} and \ref{assm:4} that the nested ratio expectation can be efficiently estimated by using the antithetic multilevel Monte Carlo estimator. Plugging this result into a slightly generalized version of \cite[Theorem~3]{GG19}, our antithetic multilevel estimator is shown to achieve a root-mean-square accuracy $\varepsilon$ at a cost of optimal $O(\varepsilon^{-2})$. As mentioned in Remark~\ref{rem:non-antithetic}, without the antithetic technique, we can only expect a suboptimal result for the overall computational complexity. Our theoretical analysis is supported by numerical experiments.

In future work, following the idea from \cite{BDM11,GH19}, we will examine the use of an adaptive number of inner samples for $\theta$ with an aim to bring about further computational savings for estimating EVSI. For the outer samples of $Y$ which are far away from the decision manifold $K$, a smaller number of inner samples than $M_02^\ell$ may be sufficient to ensure that
\[ \arg\max_{d\in D}\overline{g_d}=\arg\max_{d\in D}\overline{g_d}^{(a)}=\arg\max_{d\in D}\overline{g_d}^{(b)}=d_\opt . \]
holds with high probability. Thus it is only the outer samples of $Y$ which are close to $K$ that require great accuracy for estimating the inner conditional expectation. In this way we expect to reduce the value of $\gamma$ while maintaining the fast decay of $\EE[|\Delta P_\ell|]$ and $\VV[\Delta P_\ell]$ as derived in Theorem~\ref{thm:main}. This should reduce the total computational cost by some constant factor, approximately independent of the desired accuracy.

\section*{Acknowledgements}
The authors would like to thank Abdul-Lateef Haji-Ali of Heriot-Watt University for useful discussions and comments on the earlier version of this article. HT was supported by the Medical Research Council grant number MR/S036709/1.

\bibliographystyle{siamplain}
\bibliography{references}

\end{document}